\documentclass[12pt]{article}
\usepackage{amsmath,amsthm,bbm,amssymb}

\usepackage[all]{xy}

\usepackage[active]{srcltx}

\newtheorem*{theorem}{Theorem}
\newtheorem*{lemma}{Lemma}

\theoremstyle{definition}
\newtheorem*{example}{Example}

\theoremstyle{remark}

\newcommand{\CC}{\mathbb C}

\DeclareMathOperator{\rank}{rank}

\newcommand{\lin}{\,\frac{\quad\ }{\quad\ }\,}
\newcommand{\arr}{\longrightarrow}

\begin{document}
\title{Topological classification of systems of bilinear and sesquilinear forms\thanks{Published in: Linear Algebra Appl. 515 (2017) 1--5.}}

\author{Carlos M. da Fonseca\thanks{Department of Mathematics, Kuwait University, Kuwait, {carlos@sci.kuniv.edu.kw}}
\and
Vyacheslav Futorny\thanks{Department of Mathematics, University of S\~ao Paulo, Brazil, {futorny@ime.usp.br}
}
\and
Tetiana Rybalkina\thanks{Institute of Mathematics, Tereshchenkivska 3, Kiev, Ukraine, {rybalkina\_t@ukr.net}}
\and
Vladimir~V.~Sergeichuk\thanks{Institute of Mathematics, Tereshchenkivska 3, Kiev, Ukraine, {sergeich@imath.kiev.ua}}}

\date{}
 \maketitle

\begin{abstract}
 Let $\cal A$ and $\cal B$ be two systems consisting of the same vector spaces $\CC^{n_1},\dots,\CC^{n_t}$ and bilinear or sesquilinear forms $A_i,B_i:\CC^{n_{k(i)}}\times\CC^{n_{l(i)}}\to\CC$, for $i=1,\dots,s$. We prove that $\cal A$ is transformed to $\cal B$ by homeomorphisms within $\CC^{n_1},\dots,\CC^{n_t}$ if and only if $\cal A$  is transformed to $\cal B$ by linear bijections within  $\CC^{n_1},\dots,\CC^{n_t}$.
 


\end{abstract}

\section{Introduction}

Two bilinear or sesquilinear forms $A,B:\CC^n\times \CC^n\to\CC$ are \emph{topologically equivalent} if there exists a homeomorphism (i.e., a continuous bijection whose inverse is also a continuous bijection)
$\varphi :\CC^n\to\CC^n$ such that $A(u,v)=B(\varphi u,\varphi v)$ for all $u,v\in\CC^n$. We prove that \emph{two forms are topologically equivalent if and only if they are equivalent}; we extend this statement to arbitrary systems of forms.
Therefore, the  canonical matrices of bilinear and sesquilinear forms given in \cite{hor-ser_can} are also their canonical matrices with respect to topological equivalence.

Two linear operators $A,B:\CC^n\to \CC^n$ are \emph{topologically equivalent} if there exists a homeomorphism
$\varphi :\CC^n\to\CC^n$ such that $B(\varphi u)=\varphi (A u)$, for all $u\in\CC^n$.
The problem of classifying linear operators up to topological equivalence is still open. It is solved by Kuiper and Robbin \cite{Kuip-Robb, Robb} for linear operators without
eigenvalues that are roots of unity. It is studied for arbitrary operators in
\cite{Capp-conexamp, Capp-2th-nas-n<=6,Capp-big-n<6,h-p,h-p1} and other papers.
The fact that the problem of topological classification of forms is incomparably simpler is very unexpected for the authors; three of them  only reduce it to the nonsingular case in \cite{fon}.

Hans Schneider \cite{schn} studies the topological space $H^n_r$ of $n\times n$ Hermitian matrices of rank $r$ and proves that its connected components coincide with the *congruence classes. The closure graphs for the congruence classes of $2\times 2$ and $3\times 3$ matrices and for the *congruence classes of $2\times 2$ matrices are constructed in \cite{dmy,fun 2x2}.
The problems of topological classification of matrix pencils and chains of linear mappings are studied in \cite{f-r-s,ryb+ser}.

\section{Form representations of mixed graphs}

Let $G$ be a \emph{mixed graph}; that is, a graph that may have undirected and directed edges (multiple edges and loops are allowed);  let $1,\dots,t$ be its vertices. Its \emph{form representation} $\cal A$ of dimension $\underline n=(n_1,\dots,n_t)$ is given by assigning to each vertex $i$ the vector space $\CC^{n_i}:=\CC\oplus\dots\oplus\CC$ ($n_i$ times), to each undirected edge $\alpha :i\lin j$ ($i\le j$) a bilinear form $A_{\alpha }: \CC^{n_i}\times \CC^{n_j}\to\CC$, and to each directed edge $\beta :i\arr j$ a sesquilinear form $A_{\beta }: \CC^{n_i}\times \CC^{n_j}\to\CC$ that is linear in the first argument and semilinear (i.e., conjugate linear) in the second.
Two form representations $\cal A$ and $\cal B$ of dimensions $\underline n$ and $\underline m$ are \emph{topologically isomorphic} (respectively, \emph{linearly isomorphic})
if there exists a
family of homeomorphisms (respectively, linear bijections)
\begin{equation}\label{nbj}
\varphi_1:\CC^{n_1}\to \CC^{m_1},\ \dots,\
\varphi_t:\CC^{n_t}\to \CC^{m_t}
\end{equation}
that transforms $\cal A$ to $\cal B$; that is,
\begin{equation}\label{hyu}
A_{\alpha }(u,v)=B_{\alpha }(\varphi_iu,\varphi_jv),\qquad
u\in\CC^{n_i},\ v\in\CC^{n_j}
\end{equation}
for each edge $\alpha :i\lin j$ or $i\arr j$.

\begin{example}
A form representation
\begin{equation}\label{jst}
    {\xymatrix@C=45pt{\mathcal A:&
 {\CC^{n_1}} \save !<-2.5mm,0cm>
\ar@(ul,dl)@{-}_{A_{\alpha }}
\restore
  \ar@{-}@<0.4ex>[r]^{A_{\beta }}
 \ar@{<-}@<-0.4ex>[r]_{A_{\gamma }}
 &{\CC^{n_2}}
 \save !<2.5mm,0cm>
\ar@{<-}@(ur,dr)^{A_{\delta }}
\restore }}
\end{equation}
of dimension $\underline n=(n_1,n_2)$
of the mixed graph
\begin{equation}\label{jsy}
 \xymatrix@C=45pt{G:&
 {1} \ar@(ul,dl)@{-}_{\alpha }
 \ar@{-}@<0.4ex>[r]^{\beta  }
 \ar@<-0.4ex>@{<-}[r]_{\gamma }
 &{2}
\ar@{<-}@(ur,dr)^{\delta} }
\end{equation}
is a system consisting of the
vector spaces $\CC^{n_1},\CC^{n_2}$, bilinear forms $A_{\alpha }: \CC^{n_1}\times \CC^{n_1}\to\CC$, $A_{\beta }: \CC^{n_1}\times \CC^{n_2}\to\CC$, and sesquilinear forms
$A_{\gamma }: \CC^{n_2}\times \CC^{n_1}\to\CC$, $A_{\delta }: \CC^{n_2}\times \CC^{n_2}\to\CC$. Two form representations $\mathcal A$ and $\mathcal B$ are topologically isomorphic (respectively, linearly isomorphic)
if there exist  homeomorphisms (respectively, linear bijections)
$\varphi_1$ and
$\varphi_2$ that transform $\mathcal A$ to $\mathcal B$:
\[
\xymatrix@C=45pt{\mathcal A:&
 {\CC^{n_1}} \save !<-2.5mm,0cm>
\ar@(ul,dl)@{-}_{A_{\alpha }}
\restore
\ar[d]_{\varphi _1}
  \ar@{-}@<0.4ex>[r]^{A_{\beta }}
 \ar@{<-}@<-0.4ex>[r]_{A_{\gamma }}
 &{\CC^{n_2}}
 \save !<2.5mm,0cm>
\ar@{<-}@(ur,dr)^{A_{\delta }}
\restore
\ar[d]^{\varphi _2}
                 \\
\mathcal B:&{\CC^{m_1}} \save !<-2.5mm,0cm>
\ar@(ul,dl)@{-}_{B_{\alpha }}
\restore
  \ar@{-}@<0.4ex>[r]^{B_{\beta }}
 \ar@{<-}@<-0.4ex>[r]_{B_{\gamma }}
 &{\CC^{m_2}}
 \save !<2.5mm,0cm>
\ar@{<-}@(ur,dr)^{B_{\delta }}
\restore
}
\]
\end{example}

\begin{theorem}\label{ttk}
Two form representations of a mixed graph are topologically isomorphic if and only if they are linearly isomorphic.
\end{theorem}

\section{Proof of the theorem}

\begin{lemma}\label{vfi}
If $\varphi: \CC^n\to\CC^m$ is a homeomorphism, then $n=m$ and there exists a basis $u_1,\dots,u_n$ of $\CC^n$ such that $\varphi u_1,\dots,\varphi u_n$ is also a basis of $\CC^n$.
\end{lemma}

\begin{proof}
Let $\varphi: \CC^n\to\CC^m$ be a homeomorphism. By
\cite[Section 11]{McCl}, $n=m$. Let $k\in\{1,\dots,n-1\}$. Reasoning by induction, we suppose that there exist linearly independent vectors $u_1,\dots,u_k\in\CC^n$  such that $v_1:=\varphi u_1,\dots,v_k:=\varphi u_k$ are also linearly independent. Take $u\in\CC^n$ such that $u_1,\dots,u_k,u$ are linearly independent, and write $v:=\varphi u$. If $v_1,\dots,v_k,v$ are linearly independent, we set $u_{k+1}:=u$.

Let
$v_1,\dots,v_k,v$ be linearly dependent. Take a nonzero $w\in\CC^n$ that is orthogonal to $v_1,\dots,v_k$. Then $v_1,\dots,v_k,v+aw$ are  linearly independent for each nonzero $a\in\CC$. Write $u(a):=\varphi ^{-1}(v+aw)$ and consider the matrix $M(a)$ with columns $u_1,\dots,u_k,u(a)$. Since $\rank M(0)=k+1$, there is a $(k+1)\times(k+1)$ submatrix $N(a)$ of $M(a)$ such that $\det N(0)\ne 0$. The determinant of $N(a)$ is a continuous function of $a$. Hence there is a nonzero $b\in\CC$ such that $\det N(b)\ne 0$.
Then $\rank M(b)=k+1$. We take $u_{k+1}:=u(b)$ and obtain linearly independent vectors $u_1,\dots,u_{k+1}$ such that $\varphi u_1,\dots,\varphi u_{k+1}$ are also linearly independent.

We repeat this construction until we obtain the required $u_1,\dots,u_n$.
\end{proof}

\begin{proof}[Proof of the theorem]
Let $\mathcal A$ and $\mathcal B$ be two form representations of dimensions $\underline n$ and $\underline m$  of a mixed graph $G$.

If $\mathcal A$ and $\mathcal B$ are linearly isomorphic, then they are topologically isomorphic since $n_1=m_1,\dots,n_t=m_t$ and each linear bijection $\varphi_i: \mathbb C^{n_i}\to\mathbb C^{n_i}$ is a homeomorphism.

Let $\mathcal A$ and $\mathcal B$ be topologically isomorphic via a family of homeomorphisms \eqref{nbj}.  By the lemma,
$n_1=m_1,\dots,n_t=m_t$, and for each vertex $i\in\{1,\dots,t\}$ there exists a basis $u_{i1},\dots,u_{in_i}$ of $\CC^{n_i}$ such that $v_{i1}:=\varphi u_{i1},\dots,v_{in_i}:=\varphi u_{in_i}$
is also a basis of $\CC^{n_i}$. By \eqref{hyu}, for each edge $\alpha :i\lin j$ or $i\arr j$ we have
\[
A_{\alpha }(u_{ik},u_{jl})=B_{\alpha }(v_{ik},v_{jl}),\qquad k=1,\dots,n_i,\ l=1,\dots, n_j.
\]
Hence the matrix of $A_{\alpha }:\CC^{n_i}\times \CC^{n_j}\to\CC$ in the bases  $u_{i1},\dots,u_{in_i}$ and $u_{j1},\dots,u_{jn_j}$ is equal to the matrix of $B_{\alpha }:\CC^{n_i}\times \CC^{n_j}\to\CC$ in the bases  $v_{i1},\dots,v_{in_i}$ and $v_{j1},\dots,v_{jn_j}$ (each form is fully determined by its values on the basis vectors).

For each vertex $i\in\{1,\dots,t\}$, define the linear bijection $\psi_i :\mathbb C^{n_i}\to\mathbb C^{n_i}$ such that $\psi_i u_{i1}=v_{i1},\dots,\psi_i u_{in_i}=v_{in_i}$.
The form representations $\mathcal A$ and $\mathcal B$
are linearly isomorphic via $\psi_1,\dots,\psi_t$.
\end{proof}

\section*{Acknowledgement}

C.M. da Fonseca was supported by Kuwait University Research Grant SM08/15. V.~Futorny was
supported by  CNPq grant 301320/2013-6 and by
FAPESP grant 2014/09310-5. V.V.~Sergeichuk was supported by FAPESP grant 2015/05864-9.

\end{document}